\documentclass[12pt]{amsart}
\usepackage[pagewise]{lineno}
\usepackage{amscd,amssymb}
\usepackage[arrow,matrix]{xy}
\usepackage[urlcolor=blue,linkcolor=red,colorlinks=true]{hyperref}
\usepackage{mathdots}
\usepackage{mathrsfs}
\usepackage{amsthm}
\usepackage{bm}
\usepackage{amscd}

\theoremstyle{plain}
\newtheorem{thm}[subsection]{Theorem}
\newtheorem{lem}[subsection]{Lemma}
\newtheorem{prop}[subsection]{Proposition}
\newtheorem{cor}[subsection]{Corollary}

\theoremstyle{definition}
\newtheorem{rk}[subsection]{Remark}

\newcommand{\bb}{\mathbb}

\newcommand{\p}{\partial}

\newcommand{\e}{{\bf e}}

\begin{document}
	\date{}
		
	\title[On homogeneous polynomials determined by partial derivatives]{On homogeneous polynomials determined by their partial derivatives}
		
	\author[ZHENJIAN WANG]{ ZHENJIAN WANG  }
	\address{YMSC, Tsinghua University, 100084 Beijing, China}
	\email{wzhj@mail.tsinghua.edu.cn}
		
		
		
	\begin{abstract}
	We prove that a generic homogeneous polynomial of degree $d$ is determined, up to a nonzero constant multiplicative factor, by the vector space spanned by its partial derivatives of order $k$ whenever $k\leq\frac{d}{2}-1$.
	\end{abstract}
	\maketitle
	
\section{Introduction}
 We investigate in this note the reconstructibility of a homogeneous polynomial from its partial derivatives. The study can date back to J.~Carlson and Ph.~Griffiths, who in \cite{CG} showed that a generic homogeneous polynomial could be reconstructed, up to a nonzero constant multiple, from its Jacobian ideal, or equivalently, from its first order partial derivatives; in that paper, they used this result to study variation of Hodge structures and proved the global Torelli theorem for hypersurfaces. For further developments of the determination of a homogeneous polynomial by its Jacobian ideal, see \cite{ZW} and references therein.

 In the classical theory of variation of Hodge structures for smooth hypersurfaces, as in \cite{CG}, only first order derivatives of the defining homogeneous polynomials are involved. We can also construct higher order versions of this classical theory; see for instance \cite{DGI}. In this higher order analogous theory, a problem arises concerning the reconstructibility of a homogeneous polynomial from its higher order partial derivatives. In this paper, we will solve this problem and prove that a generic homogeneous polynomial has the desired property.

Let $S=\bb{C}[x_0,x_1,\cdots,x_n]$ be the graded polynomial ring in $n+1$ variables with coefficients in $\bb{C}$
$$
S=\bigoplus_{d=0}^\infty S_{n,d},
$$
where $S_{n,d}$ is the vector space of homogeneous polynomials of degree $d$. Given $f\in S_{n,d}$ and a natural number $k\geq 0$. Denote by $J_k(f)$ the graded ideal of $S$ generated by all partial derivatives of $f$ of order $k$ and by $E_k(f)$ the degree $d-k$ homogeneous component of $J_k(f)$, that is, the vector space spanned by all $k$-th order partial derivatives of $f$.
We will prove the following theorem.

\begin{thm}\label{thm2}
Given $n\geq1$ and $d\geq 3$, and $k\geq 1$ a natural number such that $k\leq \frac{d}{2}-1$. Suppose $f$ is a generic  homogeneous polynomial in $S_{n,d}$.

Let $g$ be another homogeneous polynomial in $S_{n,d}$ such that $E_k(f)=E_k(g)$, then $g\in\bb{C}^*f$.
\end{thm}

The underlying idea in the proof is very simple, so we give an outline here. We will show that $E_{k-1}(g)=E_{k-1}(f)$, then apply induction on $k$ to obtain $E_r(g)=E_r(g)$ for all $0\leq r\leq k$. Since $E_0(f)$ is essentially nothing but $\bb{C}f$, the conclusion follows immediately.

Note that we already have a more precise result, Theorem 1.1 in \cite{ZW}, for the case $k=1$. But we do not need to use it to prove Theorem \ref{thm2}; instead, we will use induction on $k$ until the case $k=0$ is reached. In addition, the restriction $k\leq\frac{d}{2}-1$ is given in order to ensure that $\dim E_{k+1}(f)=\dim S_{n,k+1}$ for a generic $f$, see Lemma \ref{lem1} below.

\bigskip
We almost proved the above theorem when we wrote the paper \cite{ZW}, but we did not write down the complete proof due to the lack of adequate knowledge of its applications. We would like to thank Professor A. Dimca for his kindly pointing out the applications to the study of higher Jacobians, associated polar maps, and so on. We thank the referee for valuable comments.
We also thank Yau Mathematical Sciences Center for their financial support and wonderful working atmosphere.

\section*{Notations}
As in the introduction, $S_{n,d}$ denotes the vector space of homogeneous polynomials of degree $d$.

The multi-index set
$$
\bb{N}^{n+1}=\{(i_0,i_1,\cdots, i_n)\,:\,i_j\geq 0\ \text{for }  j=0,1,\cdots, n\}.
$$
We denote by $I$ an element of $\bb{N}^{n+1}$. We shall see $\bb{N}^{n+1}$ as a subset of the vector space $\bb{R}^{n+1}$; among operations on $\bb{N}^{n+1}$ are addition, subtraction and multiplication by a positive integer:
$$
I\pm I'=(i_0\pm i_0',i_1\pm i_1',\cdots, i_n\pm i_n')
$$
and
$$
mI=(mi_0,\cdots,mi_n)
$$
for $I=(i_0,\cdots, i_n), I'=(i_0',\cdots, i_n')$, and $m\in\bb{N}$.

Denote by $\e_j,j=0,\cdots, n$ the canonical basis of $\bb{R}^{n+1}$:
$$
\e_j=(0,\cdots, 0,1,0,\cdots,0),
$$
where 1 lies in the $j$-th entry. Using this basis, we may write $I=(i_0,\cdots, i_n)$ as $I=\sum_{j=0}^ni_j\e_j$.

Moreover, there is an obvious partial ordering $"\geq"$ on $\bb{N}^{n+1}$, with
$$
I=(i_0,i_1,\cdots,i_n)\geq I'=(i_0',\cdots,i_n')\Leftrightarrow i_j\geq i_j',\quad j=0,\cdots, n,
$$
or more concisely,
$$
I\geq I'\Leftrightarrow I-I'\in\bb{N}^{n+1}.
$$

The order of $I=(i_0,\cdots, i_n)$:
$$
|I|=i_0+\cdots+i_n.
$$

For $f\in S_{n,d}$, the partial derivative of $f$ of type $I$ is
$$
D_If=\frac{\p^{|I|}f}{\p x_0^{i_0}\p x_1^{i_1}\cdots\p x_n^{i_n}}.
$$

By definition, $E_k(f)$ is the vector subspace of $S_{n,d-k}$ spanned by $D_If, |I|=k$; thus we have
$$
E_k(f)=\langle D_If\,:\,|I|=k\rangle.
$$

\section{Polynomials determined by higher order derivatives}
In this section, we will give the proof of Theorem \ref{thm2}.

We begin our proof with the following lemma.

\begin{lem}\label{lem2}
Let $f\in S_{n,d}$. If $k\geq 1$ and $\dim E_k(f)=\dim S_{n,k}$, then $\dim E_{k-1}(f)=\dim S_{n,k-1}$.
\end{lem}

\begin{proof}
The proof is almost obvious: if we are given a linear relation
$$
\sum_{|I|=k-1}a_ID_If=0,
$$
by taking differentiation with respect to the variable $x_0$, it follows that
$$
\sum_{|I|=k-1}a_ID_{I+\e_0}f=0.
$$
On the other hand, the assumption of $E_k(f)$ implies that $\{D_{I+\e_0}f\,:\,|I|=k-1\}$ are linearly independent, so $a_I=0$ for all $I$.
\end{proof}

An induction on $k$ gives the following corollary.

\begin{cor}\label{cor1}
Let $f\in S_{n,d}$. If $k\geq 1$ and $\dim E_k(f)=\dim S_{n,k}$, then $\dim E_r(f)=\dim S_{n,r}$ for all $0\leq r\leq k$.
\end{cor}

As a second step to the proof of Theorem \ref{thm2}, we show the following proposition.

\begin{prop}\label{prop1}
Given $n\geq 1$, $d\geq 3$, and $k\geq 1$. Let $f,g\in S_{n,d}$ be such that $E_k(g)=E_k(f)$ and $\dim E_{k+1}(f)=\dim S_{n,k+1}$, then $E_{k-1}(g)=E_{k-1}(f)$.
\end{prop}

\begin{proof}
We will show $E_{k-1}(g)\subseteq E_{k-1}(f)$. This is sufficient for our purpose because the two vector spaces have the same dimension by Corollary \ref{cor1}.

From $E_k(g)=E_k(f)$, we have a system of linear relations as follows: for all $I\in\bb{N}^{n+1}$ such that $|I|=k$, we have
\begin{equation}\label{eq2}
D_Ig=\sum_{|I'|=k}a_{I,I'}D_{I'}f.
\end{equation}
for some $a_{I,I'}\in\bb{C}$.

Our discussions in the sequel will be divided into two steps.

{\bf Step 1: Differentiating equations. }
Fix $I$ and $0\leq p\leq n$ such that $I\geq \e_p$. For any $0\leq q\leq n$, we will apply the equality $D_{\e_q}D_Ig=D_{\e_p}D_{(I-\e_p)+\e_q}g$ to equation \eqref{eq2}; to this end, we obtain first
\begin{eqnarray*}
D_{\e_q}D_Ig&=&D_{\e_q}\biggl(\sum_{|I'|=k}a_{I,I'}D_{I'}f\biggr)\\
            &=&\sum_{|I'|=k}a_{I,I'}D_{I'+\e_q}f;\\
\end{eqnarray*}
second,
\begin{eqnarray*}
D_{\e_p}D_{(I-\e_p)+\e_q}g&=&D_{\e_p}\biggl(\sum_{|I'|=k}a_{(I-\e_p)+\e_q,I'}D_{I'}f\biggr)\\
            &=&\sum_{|I'|=k}a_{I-\e_p+\e_q,I'}D_{I'+\e_p}f.
\end{eqnarray*}
From our assumption $\dim E_{k+1}(f)=\dim S_{n,k+1}$, it follows that $\{D_Jf\,:\, |J|=k+1\}$ are linearly independent. Therefore, using $D_{\e_q}D_Ig=D_{\e_p}D_{(I-\e_p)+\e_q}g$ and comparing the coefficients of each term $D_Jf$, we obtain that
$$
a_{I,J-\e_q}=a_{I-\e_p+\e_q,J-\e_p}
$$
for all $|J|=k+1$. Here we used the convention that $a_{I,J-\e_q}=0$ if $J\not\geq \e_q$.

Since the above conclusion holds for all $I,J,p,q$ satisfying $I\geq\e_p$, it follows that for all $I,I',p,q$ such that $|I|=|I'|=k$ and $I\geq\e_p$,
\begin{equation}\label{eq3}
a_{I,I'}=a_{I-\e_p+\e_q,I'-\e_p+\e_q}.
\end{equation}

{\bf Step 2: Considering $(k-1)$-th order partial derivatives. }
Let $K\in\bb{N}^{n+1}$ be such that $|K|=k-1$, then the Euler formula for $D_Kg$ gives
\begin{eqnarray}\label{eq4}
(d-k+1)D_Kg=\sum_{p=0}^nx_pD_{K+\e_p}g.
\end{eqnarray}
Substituting \eqref{eq2} into \eqref{eq4}, we have
$$
(d-k+1)D_Kg=\sum_{p=0}^n\sum_{|I'|=k}x_pa_{K+\e_p,I'}D_{I'}f.
$$
By \eqref{eq3}, we deduce first of all that $a_{K+\e_p,I'}=0$ if $I'\not\geq \e_p$, and thus
$$
(d-k+1)D_Kg=\sum_{p=0}^n\sum_{I'\geq\e_p}x_pa_{K+\e_p,I'}D_{I'}f=\sum_{p=0}^n\sum_{I'\geq\e_p}x_pa_{K+\e_p,(I'-\e_p)+\e_p}D_{I'}f,
$$
or written in a more convenient way,
\begin{eqnarray*}
(d-k+1)D_Kg&=&\sum_{p=0}^n\sum_{|K'|=k-1}x_pa_{K+\e_p,K'+\e_p}D_{K'+e_p}f\\
           &=&\sum_{|K'|=k-1}\biggl(\sum_{p=0}^nx_pa_{K+\e_p,K'+\e_p}D_{K'+e_p}f\biggr).
\end{eqnarray*}
Now the relations \eqref{eq3} imply that $a_{K+\e_p,K'+\e_p}=a_{K+\e_q,K'+\e_q}$ for any $p,q=0,\cdots, n$; therefore, we obtain
$$
(d-k+1)D_Kg=\sum_{|K'|=k-1}a_{K+\e_0,K'+\e_0}\biggl(\sum_{p=0}^nx_pD_{K'+e_p}f\biggr).
$$
By the Euler formula for $D_{K'}f$, we have that
$$
\sum_{p=0}^nx_pD_{K'+e_p}f=(d-k+1)D_{K'}f,
$$
 so
$$
D_Kg=\sum_{|K'|=k-1}a_{K+\e_0,K'+\e_0}D_{K'}f.
$$
Since this holds for all $K$ satisfying $|K|=k-1$, it follows that $E_{k-1}(g)\subseteq E_{k-1}(f)$.
\end{proof}

\subsection{Linear independence of partial derivatives}
As a final step to our proof of Theorem \ref{thm2}, we need the following lemma, which is interesting in its own right; see also \cite{DGI}, Proposition 3.4.

\begin{lem}\label{lem1}
Given $n\geq 1$ and $d\geq 3$, suppose $0\leq k\leq\frac{d}{2}$. Then for a generic $f\in S_{n,d}$, we have
$$
\dim E_k(f)=\dim S_{n,k}.
$$
\end{lem}

\begin{proof}
In fact, this result has already been proved in \cite{IE} as well as in \cite{Ia}. For completeness, we give a detailed proof in our setting.
Suppose given a linear relation
\begin{equation}\label{eq1}
\sum_{|I|=k}a_ID_If=0.
\end{equation}

Let $R=\bb{C}[y_0,\cdots,y_n]$ be the polynomial ring in variables $y_0,\cdots, y_n$. The apolar action of $R$ on $S$ is given by
$$
R\times S\to S, (F,f)\mapsto\langle F,f\rangle=F\bigl(\frac{\p}{\p x_0},\cdots,\frac{\p }{\p x_n}\bigr)f(x_0,\cdots, x_n);
$$
see \cite{Ia}.
Denote
$$
P(y_1,\cdots,y_n)=\sum_{|I|=k}a_I y^I,
$$
then the relation \eqref{eq1} can be simply written as $\langle P,f\rangle=0$.

From \cite[Proposition 3.4]{Ia}, it follows that for a generic $f$, the map $R_k \to S_{d-k}, Q\mapsto\langle Q,f\rangle$ is injective since $k\leq d/2$ implies that $\dim S_{d-k}\geq\dim R_k$. Hence from relation \eqref{eq1}, we obtain $P=0$ when $f$ is generically chosen, i.e., $a_I=0$ for all $I$ with $|I|=k$. We are done.
\end{proof}

\begin{rk}
In view of the obvious bound for $\dim E_k(f)$ given by
$$
\dim E_k(f)\leq\min\{\dim S_{n,k},\dim S_{n,d-k}\},
$$
the condition on $k$ in Lemma \ref{lem1} is optimal.
\end{rk}

\subsection{Proof of Theorem \ref{thm2}}
Let $f$ be a generic polynomial in $S_{n,d}$ and $E_k(g)=E_k(f)$. Under the assumption $k\leq\frac{d}{2}-1$, it follows that $k+1\leq\frac{d}{2}$, hence, by Lemma \ref{lem1}, we have $\dim E_{k+1}(f)=\dim S_{n,k+1}$; therefore the requirements in Proposition \ref{prop1} are satisfied. By Proposition \ref{prop1}, it follows that $E_{k-1}(g)=E_{k-1}(f)$. Note that by Corollary \ref{cor1}, we have $\dim E_k(f)=\dim S_{n,k}$, so the requirements in Proposition \ref{prop1} are satisfied with $k$ replaced by $k-1$ and we obtain $E_{k-2}(g)=E_{k-2}(f)$. These arguments can be repeated until we obtain $E_0(g)=E_0(f)$. By definition, we have $E_0(g)=\bb{C}g$ and $E_0(f)=\bb{C}f$, therefore $g$ is a constant multiple of $f$.

\section{Applications}

As pointed out in the introduction, the most remarkable application of the results in this paper lies in the study of higher order analogue of variation of Hodge structures for hypersurfaces; see \cite{DGI}. In this section, we give some other applications in the study of deformations of homogeneous polynomials.

For $k\geq 0$, denote by $\mathcal{U}_{n,d}(k)$ the following set
$$
\mathcal{U}_{n,d}(k)=\{\,f\in S_{n,d}\,:\,\dim E_k(f)=\dim S_{n,k}\}.
$$
From semi-continuity of $\dim E_k(f)$ with respect to $f$, we see that $\mathcal{U}_{n,d}(k)$ is a Zariski open subset of $S_{n,d}$. Obviously, we have $\mathcal{U}_{n,d}(k)=\emptyset$ if $k>\frac{d}{2}$. From Lemma \ref{lem1}, we have the following result.

\begin{cor}
Given $n\geq 1$ and $d\geq 3$. For $k\leq\frac{d}{2}$, the set $\mathcal{U}_{n,d}(k)$ is a Zariski open dense subset of $S_{n,d}$.
\end{cor}

In addition, for any $f\in\mathcal{U}_{n,d}(k)$, we have by definition that $\dim E_k(f)=\dim S_{n,k}$; by Lemma \ref{lem2}, we deduce that $\dim E_{k-1}(f)=\dim S_{n,k-1}$, that is $f\in\mathcal{U}_{n,d}(k-1)$. In other words, for fixed $n$ and $d$, the sequence of sets $\{\mathcal{U}_{n,d}(k)\}$ satisfies the following relations
$$
\mathcal{U}_{n,d}(0)\supseteq\mathcal{U}_{n,d}(1)\supseteq\cdots\supseteq\mathcal{U}_{n,d}(k)\supseteq\mathcal{U}_{n,d}(k+1)\supseteq\cdots.
$$

Note that $\mathcal{U}_{n,d}(k)$ is a cone in $S_{n,d}$, hence we can consider its projectivization, denoted by $\bb{P}(\mathcal{U}_{n,d}(k))$, in $\bb{P}(S_{n,d})$. Similar to the construction in \cite{ZW}, the assignment
$$
[f]\mapsto\bb{P}(E_k(f))
$$
gives a well-defined map, denoted by $\varphi_k$, from $\bb{P}(\mathcal{U}_{n,d}(k))$ to an obvious Grassmannian for $k\leq\frac{d}{2}$.

Using Proposition \ref{prop1} and Lemma \ref{lem2}, we prove the following result which gives an extension of Corollary 7.7 in \cite{ZW}.

\begin{cor}
For $k\leq\frac{d}{2}-1$, the map $\varphi_k: \bb{P}(\mathcal{U}_{n,d}(k))\ni[f]\mapsto\bb{P}(E_k(f))$ is injective when restricted to $\bb{P}(\mathcal{U}_{n,d}(k+1))$. In particular, it is generically injective.
\end{cor}

\begin{proof}
To begin the proof, suppose $[f]$ and $[g]$ are two elements of $\bb{P}(\mathcal{U}_{n,d}(k+1))$ such that $\varphi_k([f])=\varphi_k([g])$. By the definition of $\varphi_k$, this means that $E_k(f)=E_k(g)$. Now the assumption $[f]\in\bb{P}(\mathcal{U}_{n,d}(k+1))$ implies that $\dim E_{k+1}(f)=\dim S_{n,k+1}$, hence by Proposition \ref{prop1}, we obtain $E_{k-1}(f)=E_{k-1}(g)$. An induction argument on $k$ gives $[f]=[g]$, which goes exactly the same as the proof of Theorem \ref{thm2} where only the properties $\dim E_{k+1}(f)=\dim S_{n,k+1}$ and $E_k(f)=E_k(g)$ are essentially used. Thus, $\varphi_k$ is injective on $\bb{P}(\mathcal{U}_{n,d}(k+1))$.
\end{proof}

\begin{rk}
We do not know whether $\varphi_k$ is injective on $\bb{P}(\mathcal{U}_{n,d}(k))$ or not, except the case $k=\frac{d}{2}$ where $\varphi_{\frac{d}{2}}$ is a constant map, because in this case $E_k(f)=S_{n,k}$ for any $f\in\mathcal{U}_{n,d}(k)$.
\end{rk}


\begin{thebibliography}{00}




	
		\bibitem{CG} J.~Carlson, P.~Griffiths, \emph{Infinitesimal variations of Hodge structure and the global Torelli problem}, in Journ\'{e}es de g\'{e}ometrie alg\'{e}brique d'Angers, edited by A. Beauville, p. 51--76, Sijthoff and Noordhoff (1980).	

       \bibitem{DGI} A.~Dimca, R.~Gondim, and G.~Ilardi, \emph{Higher order Jacobians, Hessians and Milnor algebras}, Collect.
Math. (2019). https://doi.org/10.1007/s13348-019-00266-1
	
       \bibitem{Ia} A.~Iarrobino, \emph{Compressed Algebras: Artin Algebras Having Given Socle Degrees and Maximal Length}, Trans. Amer. Math. Soc., {\bf 285} (1984), 337--378.

       \bibitem{IE} A.~Iarrobino and J.~Emsalem, \emph{Some Zero-dimensional generic singularities: finite algebras having small tangent space}, Compositio Math. {\bf 36} (1978), 145--188.



	   \bibitem{ZW} Zhenjian Wang, \emph{On homogeneous polynomials determined by their Jacobian ideal}, Manuscripta Math. {\bf 146}(2015), 559-574.
	
	
	
	
	
	
	
\end{thebibliography}
\end{document}